\documentclass[10pt]{amsart}
\usepackage{mathrsfs,cite,amssymb}
\usepackage{graphicx}
\usepackage{tikz}
\usetikzlibrary{calc,shapes,arrows,automata,positioning}
\usepackage[all,cmtip,2cell]{xy}
\UseTwocells

\newtheorem{dfn}{Definition}[section]
\newtheorem{thm}{Theorem}[section]
\newtheorem{lem}[thm]{Lemma}
\newtheorem {pro}[thm]{Proposition}

\newtheorem{rmk}[thm]{Remark}

\title{Cross-connections and variants of the full transformation semigroup}
\author{P. A. Azeef Muhammed}
\address{Institute of Natural Sciences and Mathematics,
Ural Federal University,
Lenina 51,
620000 Ekaterinburg, Russia.}
\email{azeefp@gmail.com, a.a.parail@urfu.ru}
\thanks{The author acknowledges the financial support by the Competitiveness Enhancement Program of the Ural Federal University, Russia.}
\begin{document}
\maketitle
\begin{abstract}
Cross-connection theory propounded by K. S. S. Nambooripad describes the ideal structure of a regular semigroup using the categories of principal left (right) ideals. A variant $\mathscr{T}_X^\theta$ of the full transformation semigroup $(\mathscr{T}_X,\cdot)$ for an arbitrary $\theta \in \mathscr{T}_X$ is the semigroup $\mathscr{T}_X^\theta= (\mathscr{T}_X,\ast)$ with the binary operation $\alpha \ast \beta = \alpha\cdot\theta\cdot\beta$ where $\alpha, \beta \in \mathscr{T}_X$. In this article, we describe the ideal structure of the regular part $Reg(\mathscr{T}_X^\theta)$ of the variant of the full transformation semigroup using cross-connections. We characterize the constituent categories of $Reg(\mathscr{T}_X^\theta)$ and describe how they are \emph{cross-connected} by a functor induced by the sandwich transformation $\theta$. This lead us to a structure theorem for the semigroup and give the representation of $Reg(\mathscr{T}_X^\theta)$ as a cross-connection semigroup. Using this, we give a description of the biordered set and the sandwich sets of the semigroup.
\end{abstract}
\section{Introduction}\label{secintro}
In 1973, T. E. Hall \cite{hall} used the principal ideals and translations of a regular semigroup to study its ideal structure. In 1974, P. A. Grillet \cite{gril,gril1,gril2} inspired by Hall's ideas, axiomatically characterized the partially ordered sets of principal ideals of a regular semigroup. He explicitly described the relationship between the principal ideals of a regular semigroup using a pair of maps which he called a \emph{cross-connection} and gave a fundamental representation of a regular semigroup as a cross-connection semigroup. Later in 1994, K. S. S. Nambooripad \cite{cross} extended Grillet's construction to arbitrary regular semigroups (not only fundamental ones), by replacing partially ordered sets with \emph{normal categories}. Given an arbitrary regular semigroup, it induces a cross-connection between the categories of principal left and right ideals, and conversely a cross-connection between suitable normal categories gives rise to a regular semigroup. 

Cross-connection theory describes how two categories (one each from the Green $\mathscr{L}$ and $\mathscr{R}$ relations) are {connected} to form a regular semigroup. Hence, it is very suitable to study the structure of semigroups which have a rather complicated ideal structure. But the theory in itself is quite abstract that it may not look attractive for people working in concrete semigroup problems. This article is a humble effort to fill that void, wherein we give concrete meanings to the abstract notions of the cross-connection theory, by placing them in the setting of the transformation semigroup. 

It may be noted here that the cross-connection structure degenerates in many situations like regular monoids (see Section \ref{seccxn} for details). Hence, it is not easy to find a concrete setting which demonstrates the nuances and subtleties of the sophisticated construction. Fortunately, the regular part of the {variant} of the full transformation semigroup provides such a concrete setting which is very amenable for cross-connection analysis. Here we shall see that the cross-connections are determined by the {variant} element $\theta$ and hence the pun in the title of the article is vindicated.

Recall, that the full transformation semigroup $\mathscr{T}_X$ is the semigroup of all mappings from a set $X$ to itself. It is a well known that $\mathscr{T}_X$ is regular and every semigroup can be realised as a transformation semigroup \cite{clif}. Hence the semigroup $\mathscr{T}_X$ and its subsemigroups have been studied extensively \cite{cfts,howigs,igd,mstx}. The cross-connections of the \emph{singular} transformation semigroup $Sing(X)$ was studied recently in a joint article of the author with A. R. Rajan \cite{tx}. The categories involved were characterized as the powerset category $\mathscr{P}(X)$ of all non-empty proper subsets of $X$ and the category $\Pi(X)$ of all non-identity partitions of $X$. It was shown that \emph{every} cross-connection semigroup that arises from $\mathscr{P}(X)$ and $\Pi(X)$ is isomorphic to $Sing(X)$. It can be easily shown that the similar results hold for the full transformation semigroup $\mathscr{T}_X$ as well (see Section \ref{secvar} below). These results also push us closer to the variants of $\mathscr{T}_X$.

Let $(S,\cdot)$ be an arbitrary semigroup. Then the variant $S^\theta$ for an arbitrary $\theta \in S$ is defined as the semigroup $(S,\ast)$ with the binary composition $\ast$ as follows. 
$$ \alpha \ast \beta = \alpha \cdot \theta \cdot \beta \quad \text{ for } \alpha, \beta \in S.$$
Variant of a semigroup was initially studied by K. D. Magill \cite{magill} and J. B. Hickey \cite{hickey}; later by T. A. Khan and M. V. Lawson \cite{khan}, G. Y. Tsyaputa \cite{tsya}, Y. Kemprasit \cite{kempra}, I. Dolinka and J. East \cite{igd} among others. See \cite{igd} for a detailed discussion on the development of the literature. As noted in \cite{khan}, the variants arise naturally in the context of Rees matrix semigroups. It is worth observing here that the cross-connection in a regular Rees matrix semigroup $\mathscr{M}[G;I,\Lambda;P]$ is completely determined by the sandwich matrix $P$ \cite{css}. 

It is known that even if a semigroup $S$ is regular, its variant $S^\theta$ need not be regular. But when $S$ is a regular semigroup, the regular elements of $S^\theta$ form a subsemigroup \cite{khan}. In particular, $Reg(\mathscr{T}_X^\theta)$ forms a subsemigroup and it was studied in detail recently by I. Dolinka and J. East \cite{igd}. They described the structure of the regular part $Reg(\mathscr{T}_X^\theta)$ using the map $a \mapsto (a\theta,\theta a)$ starting from the right and left translations. We shall also give a similar structural description but using the much more general theory of cross-connections. This in turn suggests that their results obtained in the specific case of $Reg(\mathscr{T}_X^\theta)$ are much more universal in nature. So, our discussion may also shed some light on the structural aspects of a more general class of semigroups. 

The structure of the article is as follows. In Section \ref{seccxn}, we briefly outline the cross-connection theory by describing how the categories of principal left (right) ideals of a regular semigroup are cross-connected. Using that we describe the representation of a semigroup as a cross-connection semigroup. A reader more interested in the concrete case of $Reg(\mathscr{T}_X^\theta)$ may skip this section as the discussion in sequel is more or less self contained. In Section \ref{secvar}, we characterize the categories involved in the construction of $Reg(\mathscr{T}_X^\theta)$ as $\mathscr{P}_\theta$ and $\Pi_\theta$. These are full subcategories of $\mathscr{P}(X)$ and $\Pi(X)$ respectively and we describe the intermediary semigroups that arises from these categories. We also characterize the \emph{normal duals} $N^*\mathscr{P}_\theta$ and $N^*\Pi_\theta$ in this section. In Section \ref{secvarcxn}, we describe how a cross-connection is induced by the sandwich element $\theta$ and characterize the cross-connection bifunctors. We describe the duality between the cross-connections using the natural isomorphism between the bifunctors. This explains how the categories are cross-connected to form $Reg(\mathscr{T}_X^\theta)$. In the next section, we use our cross-connection representation to give a description of the biorder structure of $\mathscr{T}_X$ and $Reg(\mathscr{T}_X^\theta)$, in terms of subsets and partitions. We conclude with an illustrative example.

\section{Theory of cross-connections}\label{seccxn}
In this section, we briefly describe the theory of cross-connections by starting from a regular semigroup, and reconstructing it as a cross-connection semigroup from its constituent normal categories of principal left (right) ideals. We include this rather lengthy preliminary section to exposit the cross-connection structure of a regular semigroup, considering the lack of easy availability of Nambooripad's original work \cite{cross}, and the formal presentation style in \cite{cross}, which makes it quite cryptic for a fresh reader.

The reverse process of constructing a regular semigroup from abstract normal categories is much more involved. An interested reader may refer \cite{tx} for an introductory description, \cite{kvn} for a more concise but less detailed discussion or \cite{cross} for the complete construction. See \cite{azeefcross} for some references on the theory of cross-connections.

We assume some basic notions from semigroup theory \cite{clif,grillet} and category theory \cite{bucur,mac}. In this article, all the functions are written in the order of their composition, i.e., from left to right. For a category $\mathcal C$, the set of objects of $\mathcal C$ is denoted by $v\mathcal C$ and the set of morphisms by $\mathcal{C}$ itself. Thus the set of all morphisms between objects $c,d \in v\mathcal{C}$ is denoted by $\mathcal{C}(c,d)$.

A \emph{normal category} is a specialised category whose object set is a partially ordered set and the morphisms admit suitable factorizations. In fact, a normal category was axiomatised so that the principal ideals of a regular semigroup with partial translations as morphisms formed a normal category, and conversely every normal category arose this way.  

Let $S$ be a regular semigroup. Then there are two normal categories associated with it: the principal left ideal category $\mathcal{L}$ and the principal right ideal category $\mathcal{R}$. An object of the category $\mathcal{L}$ is a principal left ideal $Se$ for $e\in E(S)$, and a morphism from $Se$ to $Sf$ is a partial right translation $\rho(e,u,f): u \in eSf$. That is, for $x\in Se$, the morphism $\rho(e,u,f):x \mapsto xu \in Sf$. Dually, the objects of the category $\mathcal{R}$ are the principal right ideals $eS$, and the morphisms are partial left translations $\lambda(e,w,f) : w \in fSe$. 

We mention in passing that the principal left ideal category $\mathcal{L}$ is category isomorphic to the \emph{Karoubi envelope} of a semigroup introduced by B. Tilson in connection with the Delay Theorem \cite[Section 17]{tilson}, and later studied by A. Costa and B. Steinberg in relation to the Sch{\"u}tzenberger category of a semigroup \cite{costa}.

Observe that we can have a partial order on the object set of $\mathcal{L}$, namely the order induced by set inclusions. So, if $Se \subseteq Sf$, we can see that we have an \emph{inclusion} morphism $\rho(e,e,f)= j(Se,Sf)$ from $Se$ to $Sf$. The objects of the category $\mathcal{L}$ along with inclusion morphisms form a strict-preorder subcategory $\mathcal{P}$ of the category $\mathcal{L}$. This is the category of inclusions in $\mathscr{L}$ (also see  \cite{rajancat}). 

Given an inclusion $\rho(e,e,f)$, we have a morphism $\rho(f,fe,e):Sf \to Se$ such that $\rho(e,e,f) \rho(f,fe,e) = \rho(e,efe,e) = \rho(e,e,e) = 1_{Se}$. Therefore, we say that the inclusion $\rho(e,e,f)$ {splits}, and its right inverse $\rho(f,fe,e)$ shall be called a \emph{retraction}. A morphism $\rho(e,u,f)$ will be an \emph{isomorphism} if it has both a right inverse and a left inverse, which happens when $e \mathscr{D}f$.  

Given a morphism $\rho = \rho(e,u,f)$ in $\mathcal{L}$, for any $g\in E(R_{u})$ such that $eg=ge=g$ and $h\in E(L_{u})$, we can factorize 
$$ \rho(e,u,f) = \rho(e,g,g)\rho(g,u,h)\rho(h,h,f), $$
where $\rho(e,g,g)$ is a retraction, $\rho(g,u,h)$ is an isomorphism and $\rho(h,h,f)$ is an inclusion. Such a factorization is called a \emph{normal factorization} of $\rho(e,u,f)$ in $\mathcal{L}$. The morphism $\rho(e,gu,h): Se\to Sh$ is known as the {epimorphic component} $\rho^\circ$ of the morphism $\rho$. The codomain $Sh$ of $\rho^\circ$ is called the {image} of the morphism $\rho$ denoted by Im $\rho$.
The following diagram illustrates the factorization property of a morphism in $\mathcal{L}$.
\begin{equation*}\label{normalfactls}
\xymatrixcolsep{6pc}\xymatrixrowsep{6pc}\xymatrix
{
 Se \ar[r]^{\rho} \ar@<-3pt>[d]_{\rho(e,g,g)}   & Sf  \\       
 Sg \ar@{.>}@<-3pt>[u]_{j(Sg,Se)} \ar@<-3pt>[r]_{\rho(g,u,h)} & Sh\ar@{.>}@<-3pt>[l]_{\rho(h,u,g)}\ar[u]_{\rho(h,h,f)= j(Sh,Sf)} 
}
\end{equation*}

Now we proceed to describe {normal cones} in the category $\mathcal{L}$. These are the basic building blocks of our construction since the cross-connection semigroup we obtain eventually will consist of ordered pairs of normal cones. A normal cone is essentially a collection of {`nice'} morphisms with a distinguished vertex in $\mathcal{L}$. 
\begin{dfn}
A \emph{normal cone} $\gamma$ with a vertex $Sd$ is a function from the object set $v\mathcal{L}$ to the set of morphisms in $\mathcal{L}$ such that  
\begin{enumerate}
\item $\gamma(Se) \in \mathcal{L}(Se,Sd)$ for all $Se \in v\mathcal{L}$;  
\item whenever $Se\subseteq Sf$ then $j(Se,Sf)\gamma(Sf) = \gamma(Se)$;
\item there exists $Sh\in v\mathcal{L}$ such that $\gamma(Sh)\colon Sh\to Sd$ is an isomorphism.
\end{enumerate}
\end{dfn}
The following diagram illustrates a normal cone in the category $\mathcal{L}$.
\begin{equation*}\label{conediag}
\xymatrixrowsep{8pc}\xymatrixcolsep{2pc}\xymatrix
{
&&Sd\\
Se\ar[r]^\subseteq_{\rho(e,e,f)} \ar[rru]|-{\gamma(Se)} &
Sf \ar[ru]|-{\gamma(Sf)} &
Sd \ar@{--}[l] \ar[u]|-{\gamma(Sd)} &
Sh \ar@{--}[l] \ar[lu]|-{\gamma(Sh)} &
Sg \ar@{--}[l] \ar[llu]|-{\gamma(Sg)}
}
\end{equation*}
The morphism $\gamma(Se)\colon  Se \to Sd$ is called the {component} of $\gamma$ at $Se$. The {M-set} of a cone $\gamma$ in $\mathcal{L}$ with vertex $Sd$ is defined as 
$$ M\gamma = \{ Se \in \mathcal{L}: e \mathscr{D} d \}. $$ 
Observe that, given a normal cone $\gamma$ with vertex $Sd$ and a morphism $\rho:Sd \to Sh$ with Im $\rho = Sg \subseteq Sh$, the map $\gamma*\rho^\circ \colon  Se \mapsto\gamma(Se)\rho^\circ$ from $v\mathcal{L}$ to $\mathcal{L}$ is a normal cone in the category $\mathcal{L}$ with vertex $Sg$. Hence, given two normal cones $\gamma$ and $\sigma$ in $\mathcal{L}$ with vertices $Sd$ and $Sh$ respectively, we can compose them as follows.
\begin{equation} \label{eqnsg1}
\gamma \cdot \sigma = \gamma*(\sigma(Sd))^\circ 
\end{equation} 
where $(\sigma(Sd))^\circ$ is the epimorphic component of the morphism $\sigma(Sd)$. If the morphism $\sigma(Sd) :Sd\to Sh$ has image $Sg$, then the vertex of the new cone $\gamma\cdot\sigma$ will be $Sg$, as illustrated in the diagram below. 
\begin{equation*} \label{diagsg2}
\xymatrixrowsep{1pc}\xymatrixcolsep{.5pc}\xymatrix
{&&&&\gamma &&&&&\cdot &&&&& \sigma &&&&&&&= &&\gamma*(\sigma(d))^\circ }
\end{equation*} \vspace{-.5cm}
\begin{equation*}\label{conecompose2}
\xymatrixrowsep{3pc}\xymatrixcolsep{.5pc}\xymatrix
{
&&Sd && && && Sh && && && Sh\\
&&&&&  \cdot &&&&&&& = && Sg\ar@{.>}[u]|-{j(Sg,Sh)}\\
Se\ar[r]_{j_1} \ar[rruu]|-{\gamma(Se)} &
Sf \ar[ruu] &
Sd \ar@{--}[l] \ar[uu]|-{\gamma(Sd)} &
Sh \ar@{--}[l] \ar[luu] &
Sg \ar[l]^{j_2} \ar[lluu]|-{\gamma(Sg)} &&
Se \ar[r]_{j_1} \ar[rruu]|-{\sigma(Se)}&
Sf \ar[ruu] & 
Sd \ar@{--}[l] \ar[uu]|-{\sigma(Sd)} &
Sh \ar@{--}[l] \ar[luu] &
Sg \ar[l]^{j_2} \ar[lluu]|-{\sigma(Sg)}&& 
Se\ar@{.>}[r] \ar@{.>}[rruu]|-{\sigma(Se)} &
Sf \ar@{.>}[ruu] \ar@{.}[r]&
Sd\ar[u]|-{(\sigma(Sd))^\circ}&
Sh\ar@{.}[l] \ar@{.>}[luu] &
Sg \ar@{.>}[l] \ar@{.>}[lluu]|-{\sigma(Sg)}\\
\\
&&&&&&&&&&&& Se\ar[r]_{j_1} \ar[rruu]|-{\gamma(Se)} &
Sf \ar[ruu] &
Sd \ar@{--}[l] \ar[uu]|-{\gamma(Sd)} &
Sh \ar@{--}[l] \ar[luu] &
Sg \ar[l]^{j_2} \ar[lluu]|-{\gamma(Sg)} &&
}
\end{equation*}
All the {normal cones} in the normal category $\mathcal{L}$ with this special binary composition form a regular semigroup $T\mathcal L$ known as the \emph{semigroup of normal cones} in $\mathcal {L}$. It can be shown that the $\mathcal{L}$ category associated with the regular semigroup $T\mathcal{L}$ is isomorphic to $\mathcal{L}$.

Now we describe some important normal cones in $\mathcal{L}$. These distinguished cones $\rho^a$ with vertex $Sa$ called \emph{principal cones}, are those induced by an element $a \in S$. The component of the cone $\rho^a$ at any $Se \in v\mathcal{L}$ is given by $\rho^a(Se) = \rho(e,ea,f) $, where $f \in E(L_a) $.
The mapping $a \mapsto \rho^a$ is a homomorphism from $S$ to $T\mathcal{L}$. Further if $S$ is a regular monoid, then $S$ is isomorphic to $T\mathcal{L}$. This is a crucial fact which shall be elaborated later in this section.

Further, Nambooripad invented a notion of \emph{normal dual} of a normal category extending Grillet's idea of the dual of a partially ordered set. It is well known that the category of all functors from a category $\mathcal{L}$ to the category $\mathbf{Set}$ with natural transformations as morphisms forms a dual category $\mathcal{L}^*$. Given the normal category $\mathcal{L}$, Nambooripad identified a {full} subcategory of $\mathcal{L}^*$ as the normal dual $N^*\mathcal{L}$ of $\mathcal{L}$.

Instead of considering all functors from $\mathcal{L}$ to $\mathbf{Set}$, Nambooripad restricted the object set of $N^*\mathcal{L}$ to certain special functors called \emph{H-functors}. 
For each $\gamma \in T\mathcal{L}$ with vertex $Sd$, the \emph{H-functor} $H({\gamma};-)\colon  \mathcal{C}\to \mathbf{Set}$ is defined as follows. For each $Se\in v\mathcal{L}$ and for each $\rho \in \mathcal{L}(Se,Sf)$, let
\begin{subequations} \label{eqnH11}
\begin{align}
H({\gamma};{Se})&= \{\gamma\ast (\varrho)^\circ :   \varrho\in \mathcal{L}(Sd,Se)\} \text{ and }\\
H({\gamma};{\rho}) :H({\gamma};{Se}) &\to H({\gamma};{Sf}) \text{ given by }\gamma\ast (\varrho)^\circ \mapsto \gamma\ast (\varrho\rho)^\circ.
\end{align}
\end{subequations}
It can be shown that the $H$-functor is a \emph{representable functor} such that for a normal cone $\gamma$ with vertex $Sd$, there is a natural isomorphism $\eta_\gamma: H(\gamma;-) \to \mathcal{L}(Sd,-)$. Here $\mathcal{L}(Sd,-)$ is the covariant hom-functor determined by $Sd\in v\mathcal{L}$. 
Also if $H({\gamma};-) = H({\gamma'};-)$, then the $M$-sets of the normal cones $\gamma$ and $\gamma'$ coincide; hence we define the $M$-set of an $H$-functor as $MH(\gamma;-) = M\gamma$.
\begin{dfn}
The \emph{normal dual} $N^\ast \mathcal{L}$ is a category with
\begin{equation} \label{eqnH1}
v N^\ast \mathcal{L} = \{ H(\epsilon;-) \: : \: \epsilon \in E(T\mathcal{L}) \}.
\end{equation} 
A morphism in $N^*\mathcal{L}$ between two $H$-functors $H(\epsilon;-)$ and $H(\epsilon';-)$ is a natural transformation $\tau$ as described in the following commutative diagram.
\begin{equation*}\label{Hfunct}
\xymatrixcolsep{4pc}\xymatrixrowsep{3pc}\xymatrix
{
 H(\epsilon;Se) \ar[r]^{\tau(Se)} \ar[d]_{H(\epsilon;\rho)}  
 & H(\epsilon';Se) \ar[d]^{H(\epsilon';\rho)} \\       
 H(\epsilon;Sf) \ar[r]^{\tau(Sf)} & H(\epsilon';Sf) 
}
\end{equation*}
\end{dfn}
Using the discussion above, the natural transformation $\tau$ in $N^*\mathcal{L}$ may be characterized as follows.
\begin{pro}\cite{cross} 
Let $\epsilon$ and $\epsilon'$ be idempotent normal cones in $\mathcal{L}$ with vertices $Sf$ and $Sg$ respectively. Then for every morphism $\tau \colon  H(\epsilon;-) \to H(\epsilon ';-)$ in $N^\ast \mathcal{L} $, there is a unique $\rho \colon  Sg \to Sf$ in $\mathcal{L}$ such that the following diagram commutes. 
\begin{equation*}\label{Homf}
\xymatrixcolsep{4pc}\xymatrixrowsep{3pc}\xymatrix
{
 H(\epsilon;-) \ar[r]^{\eta_\epsilon} \ar[d]_{\tau}  
 & \mathcal{L}(Sf,-) \ar[d]^{\mathcal{L}(\rho,-)} \\       
 H(\epsilon';-) \ar[r]^{\eta_{\epsilon'}} & \mathcal{L}(Sg,-) 
}
\end{equation*}
In this case, the component of the {natural transformation}\index{natural transformation} $\tau$ at $Se$ is the map given by:
\begin{equation} \label{eqnH2}
\tau(Se) \colon  \epsilon \ast \varrho^\circ \longmapsto \epsilon' \ast (\rho \varrho)^\circ. 
\end{equation}
\end{pro}
\begin{thm}\cite{cross} \label{thm1}
The normal dual $N^*\mathcal{L}$ forms a normal category and it is isomorphic to the $\mathcal{R}$ category associated with the regular semigroup $T\mathcal{L}$.
\end{thm}
\begin{rmk}
Dual properties hold for the category $\mathcal{R}$ of principal right ideals of $S$. The principal cones associated are represented as $\lambda^a$ given by $\lambda^a(eS) = \lambda(e,ae,f)$. There is an anti-homomorphism $a\mapsto\lambda^a$ from $S$ to $T\mathcal{R}$. The category $\mathcal{R}$ is isomorphic to the $\mathcal{L}$ category associated with the semigroup $T\mathcal{R}$, and the normal dual $N^*\mathcal{R} $ is isomorphic to $\mathcal{R}$ category associated with the regular semigroup $T\mathcal{R}$.
\end{rmk}

Thus we have four normal categories associated with a regular semigroup $S$, namely $\mathcal{L}$, $\mathcal{R}$, $N^*\mathcal{L}$ and $N^*\mathcal{R}$. Using these categories and a pair of cross-connection functors, Nambooripad explicitly described the relationship between the principal left and right ideals of the given regular semigroup.

To precisely give the relationship between the categories, we need the concept of a \emph{local isomorphism}. The use of the terminology and idea of a local isomorphism in the structure of regular semigroups, may be traced back to D. B. McAlister \cite{mcal}, where he used it to describe the structure of locally inverse semigroups. Later A. R. Rajan \cite{loc} used it to describe the local isomorphisms of Grillet's cross-connections. He observed that the order-isomorphisms of principal ideals arose from the $\omega$-isomorphisms of the biordered set. To describe this notion in the context of normal categories, we need to define an ideal of a category.

\begin{dfn}\label{ideal}
An \emph{ideal} $\langle eS \rangle$ of the category $\mathcal{R}$ is the full subcategory of $\mathcal{R}$ whose objects are the principal ideals $fS\subseteq eS$ in $\mathcal{R}$.
\end{dfn}

\begin{thm}\cite{cross}
Given a regular semigroup $S$ with normal categories $\mathcal{L}$ and $\mathcal{R}$, there is a functor $\Gamma\colon  \mathcal{R} \to N^*\mathcal{L}$ such that $\Gamma$ is inclusion preserving, fully faithful and for each $eS \in v\mathcal{R}$, $\Gamma_{|\langle eS \rangle}$ is an isomorphism of the ideal $\langle eS \rangle$ onto $\langle \Gamma(eS) \rangle$ given by
\begin{equation}
\Gamma(eS) = H(\rho^e;-) \quad\text{ and }\quad \Gamma(\lambda(e,u,f)) = \eta_{\rho^e}\mathcal{L}(\rho(f,u,e),-)\eta_{\rho^f}^{-1}.
\end{equation}
\end{thm}
Such a functor $\Gamma$ is called a \emph{local isomorphism} from $\mathcal{R}$ to $N^*\mathcal{L}$. Moreover for every $Sf \in v\mathcal{L}$, there is some $eS \in v\mathcal{R}$ such that $Sf \in M\Gamma(eS)$.
\begin{dfn}
A triplet $(\mathcal{R},\mathcal{L};\Gamma)$ is called a \emph{cross-connection} if $\Gamma$ is a \emph{local isomorphism} from $\mathcal{R}$ to $N^*\mathcal{L}$ such that for every $Sf \in v\mathcal{L}$, there is some $eS \in v\mathcal{R}$ such that $Sf \in M\Gamma(eS)$.
\end{dfn}
\begin{rmk}\label{rmkdual}
Dually, we have a \emph{dual cross-connection} $(\mathcal{L},\mathcal{R};\Delta)$ defined by the local isomorphism $\Delta\colon  \mathcal{L} \to N^*\mathcal{R}$ as follows.
\begin{equation}
\Delta(Se) = H(\lambda^e;-) \quad\text{ and }\quad \Delta(\rho(e,u,f)) = \eta_{\lambda^e}\mathcal{R}(\lambda(f,u,e),-)\eta_{\lambda^f}^{-1}.
\end{equation}
\end{rmk}
Since $\mathcal{L}$ and $\mathcal{R}$ are cross-connected with $\Gamma$ and the dual $\Delta$, by category isomorphisms \cite{mac}, we have two associated bifunctors $\Gamma(-,-)\colon  \mathcal{L}\times\mathcal{R} \to \bf{Set}$ and $\Delta(-,-)\colon  \mathcal{L}\times\mathcal{R} \to \bf{Set}$ given as follows:
\begin{subequations}
\begin{align}
\Gamma(Se,fS) &= \Gamma(fS)(Se) \\
\Gamma(\rho,\lambda) = \Gamma(fS)(\rho)\Gamma(\lambda)&(Se') = \Gamma(\lambda)(Se)\Gamma(f'S)(\rho)\\
\Delta(Se,fS) &= \Delta(Se)(fS)\\
\Delta(\rho,\lambda) = \Delta(Se)(\lambda)\Delta(\rho)&(f'S)) = \Delta(\rho)(fS)\Delta(Se')(\lambda)) 
\end{align}
\end{subequations}
for all $(Se,fS) \in v\mathcal{L}\times v\mathcal{R}$ and $(\rho,\lambda):(Se,fS) \to (Se',f'S)$.

Using the bifunctors, we obtain the following intermediary regular semigroups which are subsemigroups of $T\mathcal L$ and $T\mathcal R$ respectively,
\begin{subequations}
\begin{align}
U\Gamma = & \bigcup\: \{  \Gamma(Se,fS) : (Se,fS) \in v\mathcal{L} \times v\mathcal{R} \} \\
U\Delta = & \bigcup\: \{  \Delta(Se,fS) : (Se,fS) \in v\mathcal{L} \times v\mathcal{R} \},
\end{align}
\end{subequations}
which may be characterized as follows:
$$U\Gamma= \{\rho^a:a\in S\} \:\text{ and }\:U\Delta = \{\lambda^a:a \in S\}.$$
There is a natural isomorphism $\chi_\Gamma$ between the bifunctors $\Gamma(-,-)$ and $\Delta(-,-)$ called the \emph{duality} associated with the semigroup $S$. Using $\chi_\Gamma$, we can \emph{link} certain normal cones in $U\Gamma$ with those in $U\Delta$. Given a cross-connection $\Gamma$ with the dual $\Delta$, a cone $\gamma \in U\Gamma$ is said to be \emph{linked} to $\delta \in U\Delta$, if there is a $(Se,fS) \in v\mathcal{L} \times v\mathcal{R}$ such that $\gamma \in \Gamma(Se,fS)$ and $ \delta = \chi_\Gamma(Se,fS)(\gamma)$. The pairs of linked cones $(\gamma,\delta)$ will form a regular semigroup called the \emph{cross-connection semigroup} $\tilde{S}\Gamma$ determined by $\Gamma$. It may be shown that the linked cones are of the form $(\rho^a,\lambda^a)$, and hence 
$$ \tilde{S}\Gamma = \:\{\: (\rho^a,\lambda^a) : a\in S \} .$$
For $(\rho^a,\lambda^a),( \rho^b,\lambda^b) \in \tilde{S}\Gamma$, the binary operation is defined by 
$$ (\rho^a,\lambda^a) \circ (\rho^b,\lambda^b) = (\rho^a.\rho^b,\lambda^b.\lambda^a).$$
Then the map $a \mapsto(\rho^a,\lambda^a)$ is an isomorphism from $S$ to $\tilde{S}\Gamma$. This gives the cross-connection representation of the regular semigroup $S$ as illustrated in the diagram below.
\begin{equation*}
\xymatrixrowsep{2pc}\xymatrixcolsep{3pc}\xymatrix
{
& \xy*{S}*\cir<10pt>{}\endxy \ar@/_7pc/@{->}[ddddl]_{a\mapsto\rho^a} \ar@/^7pc/@{->}[ddddr]^{a\mapsto\lambda^a} \ar@{.}[ld] \ar@{.}[rd]\\
\mathcal{R} \ar@{->}[d]^{\Gamma} \ar@/^3pc/@{.}[ddrr] & & 
\mathcal{L} \ar@{->}[d]_{\Delta} \ar@/_3pc/@{.}[ddll]\\
N^*\mathcal{L} &\xtwocell[d]{}<>{^{\chi_\Gamma}}{\mathcal{L}\times \mathcal{R}} \ar@/_2pc/@{->}[d]_{\Gamma(-,-)}\ar@/^2pc/@{->}[d]^{\Delta(-,-)}& N^*\mathcal{R}\\
\xy*{T\mathcal{L}}*\cir<10pt>{}\endxy \:\gamma \ar@/^-1.5pc/@{<-->}[rr]_{\chi_\Gamma(-,-)} & \mathbf{Set}& \delta\: \xy*{T\mathcal{R}}*\cir<10pt>{}\endxy\\
\xy*{U\Gamma}*\cir<10pt>{}\endxy \ar@{^{(}->}@<-1pt>[u]\:\rho^a \ar@/_-1.5pc/@{<-->}[rr] & & \lambda^a\: \xy*{U\Delta}*\cir<10pt>{}\endxy \ar@{^{(}->}@<-3pt>[u]\\
&\xy*{\tilde{S}\Gamma}*\cir<10pt>{}\endxy\ar@{->>}[lu]\ar@{->>}[ru]
}
\end{equation*}

Although the construction is a bit complicated, it gives a lot of information regarding the structure of the semigroup. Recall that for a regular semigroup $S$, the map $a \mapsto\rho^a$ may be an isomorphism (for instance in the case of a regular monoid). In that case, every intermediary regular semigroup involved in the construction (denoted using circles in the above diagram) is isomorphic (or anti-isomorphic) to $S$. Then all the solid arrows in the above diagram become isomorphisms. It may be shown (using similar arguments as in \cite{tx}) that every cross-connection semigroup in this case is isomorphic to $S$. Hence, such a semigroup has a relatively simple ideal structure.

It is still an open problem to characterize the class of regular semigroups for which $S$ is isomorphic to $T\mathcal{L}$. But it certainly contains many non-monoids. For instance, this class includes
\begin{enumerate}
\item the semigroup of singular transformations on a set \cite{tx}, 
\item the semigroup of singular linear transformations on a vector space \cite{tlx}, 
\item the semigroup of singular order preserving mappings on a chain,
\item the semigroup of singular one-one partial mappings on a set, 
\item the semigroup of singular partial mappings on a set, 
\item semilattices,
\item Clifford semigroups.
\end{enumerate}
The uncited results above are due to the author. The classes (3-5) can be proved using similar arguments as in (1). Clifford semigroups case requires some work, but is not difficult; the semilattice case follows from this.

On the contrary, we can find several classes of regular semigroups which admit a {rich} cross-connection structure. Arbitrary completely regular semigroups, completely simple semigroups \cite{css}, bands, $Reg(S^\theta)$ (where $S$ is a regular semigroup) are some of those.

\section{Categories from the variants of the full transformation semigroup} \label{secvar}

Now we shift our attention to the variants of the full transformation semigroup. To discuss the normal categories arising from this semigroup, it may be helpful to look first at the full transformation semigroup $\mathscr{T}_X$. $\mathscr{T}_X$ is the monoid of \emph{all} transformations on a set $X$, and the semigroup $Sing(X)$ of all singular transformations on $X$ forms an important regular subsemigroup of $\mathscr{T}_X$. The cross-connections of $Sing(X)$ was studied in \cite{tx}. The $\mathcal{L}$ category associated was characterized as the powerset category $\mathscr{P}(X)$ of all non-empty proper subsets of $X$ with mappings as morphisms. The $\mathcal{R}$ category was characterized as the partition category $\Pi(X)$ with the set of objects $\{\bar{\pi} : \pi\text{ is a non-identity partition of }X \}$, where $\bar{\pi}$ denotes the set of all mappings from the partition $\pi$ to $X$. A morphism $\eta^*$ in $\Pi(X)$ from $\bar{\pi}_1$ to $\bar{\pi}_2$ was defined as $\eta^*: \alpha \mapsto \eta\alpha$ for every $\alpha \in \bar{\pi}_1$ where $\eta$ is mapping from $\pi_2$ to $\pi_1$. It was shown that even though a permutation $\theta$ of $X$ induces a non-trivial cross-connection, every cross-connection semigroup that arises is isomorphic to $Sing(X)$.

It is not difficult to see that by adjoining $X$ to $v\mathscr{P}(X)$ and the set of all mappings from $\{\{x\} : x\in X \}$ to $X$ to $v\Pi(X)$ we can characterize the categories in the \emph{full} transformation semigroup $\mathscr{T}_X$. Using similar arguments (in fact easier, since the \emph{largest object}-$X$ should map to $X$ under cross-connections), we can show that every semigroup that arises is isomorphic to $\mathscr{T}_X$. In the sequel, we shall denote the associated categories of $\mathscr{T}_X$, also by $\mathscr{P}(X)$ and $\Pi(X)$. Summarising, we have the following results.

\begin{thm}
$\mathscr{P}(X)$ and $\Pi(X)$ are normal categories. The semigroup $T\mathscr{P}(X)$ of all normal cones in $\mathscr{P}(X)$ is isomorphic to $\mathscr{T}_X$ and the semigroup $T\Pi(X)$ of all normal cones in $\Pi(X)$ is anti-isomorphic to $\mathscr{T}_X$. Every cross-connection semigroup that arises from these categories is isomorphic to $\mathscr{T}_X$.
\end{thm}

It may be noticed here that given a permutation $\theta$, it induces a cross-connection between the categories $\mathscr{P}(X)$ and $\Pi(X)$ \cite{tx}. But the resulting cross-connection semigroup is isomorphic to $\mathscr{T}_X$. This cross-connection semigroup is in fact the cross-connection semigroup of the variant semigroup $\mathscr{T}_X^\theta$, where $\theta$ is a permutation. This can be seen as a reflection of the fact that the variant $\mathscr{T}_X^\theta$ is isomorphic to $\mathscr{T}_X$ if and only if $\theta$ is a permutation \cite{igd}.

In \cite{igd}, Dolinka and East explored the structure of $\mathscr{T}_X^\theta$, their idempotent generated subsemigroups, $Reg(\mathscr{T}_X^\theta)$, their ideals etc. In the process, they had described the structure of the regular part $Reg(\mathscr{T}_X^\theta)$ from the right and left translations. They demonstrated using the diagram \cite{igd} below, how a regular $\mathscr{D}$-class of $\mathscr{T}_X$ breaks up to form $\mathscr{D}^\theta$ classes in $\mathscr{T}_X^\theta$. Here $\mathscr{D}^\theta$ denotes Green's $\mathscr{D}$-relation on the variant $\mathscr{T}_X^\theta$. (Similarly, $\mathscr{R}^\theta$, $\mathscr{L}^\theta$ and $\mathscr{H}^\theta$ will be used henceforth.) 
\begin{center}
\includegraphics[scale=.3]{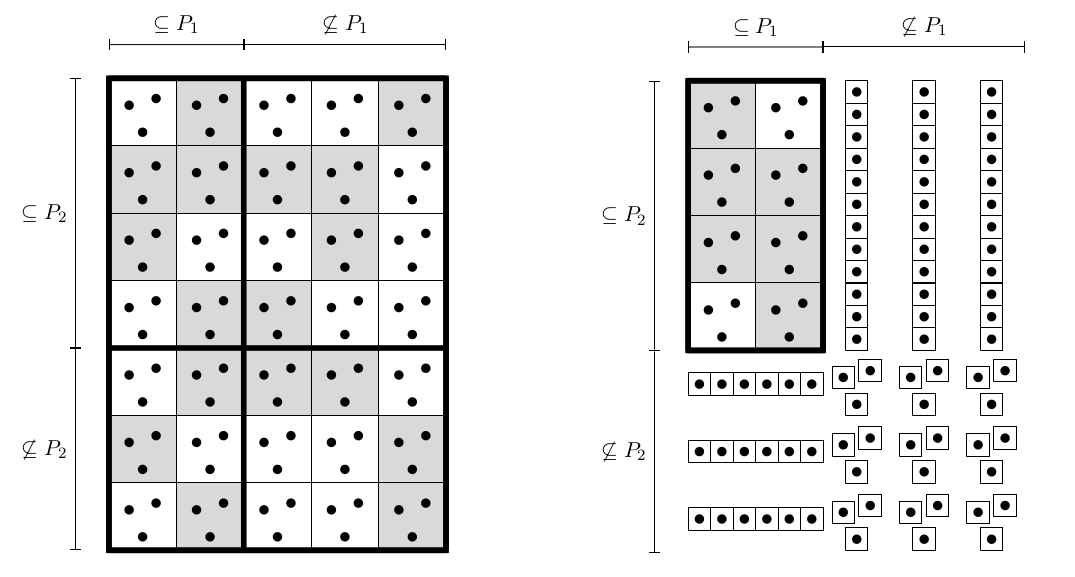}
\end{center}
The diagram on the left represents a typical $\mathscr{D}$-class of  $\mathscr{T}_X$. When the binary composition changes, it breaks up into four groups of transformations which may be described using the sets $P_1$ and $P_2$. The sets $P_1$ and $P_2$ characterized as follows are important in the sequel. 
$$P_1 = \{ a\in\mathscr{T}_X:a\theta\: \mathscr{R} \:\theta\}\text{ and }P_2 = \{ a\in\mathscr{T}_X:\theta a\: \mathscr{L} \: \theta\}$$
The first part $P_1 \cap P_2$ consists of regular elements of $\mathscr{T}_X^\theta$ and it forms a single $\mathscr{D}^\theta$ class in $\mathscr{T}_X^\theta$. The transformations which belong to $P_2\backslash P_1$ form non-regular $\mathscr{D}^\theta$-classes each consisting of a non-singleton $\mathscr{L}^\theta$-class. The $\mathscr{H}^\theta$-classes of this group are singletons. Similarly, the transformations which belong to $P_1\backslash P_2$ form non-regular $\mathscr{D}^\theta$-classes each consisting of a non-singleton $\mathscr{R}^\theta$-class. The rest of the transformations (those belonging to neither $P_1$ nor $P_2$) form non-regular singleton $\mathscr{D}^\theta$-classes.

Before we proceed to describe the categories in $Reg(\mathscr{T}_X^\theta)$, we need to fix some notations. Let $A$ be a subset of $X$, and $\alpha$, a partition (or an equivalence relation) on $X$. Borrowing the terminology from \cite{igd}, we shall say $A$ \emph{saturates} $\alpha$ if each $\alpha$-class contains at least one element of $A$. We say $\alpha$ \emph{separates} $A$ if each $\alpha$-class contains at most one element of $A$. Using this terminology, the subsets $P_1$ and $P_2$ may be described as follows.
$$P_1 = \{ a\in\mathscr{T}_X:\pi_\theta\text{ separates Im } a\}\text{ and }P_2 = \{ a\in\mathscr{T}_X:\text{Im }\theta \text{ saturates }\pi_a\}$$
It is shown that $Reg(\mathscr{T}_X^\theta) =P_1\cap P_2$. Further, the Green relations in $Reg(\mathscr{T}_X^\theta)$ are described as follows, as restrictions of the Green relations in $\mathscr{T}_X$.
\begin{pro}\cite{igd}
Let $a,b \in Reg(\mathscr{T}_X^\theta)$, then
\begin{enumerate}
\item $a \mathscr{L}^\theta b$ if and only if Im $a= $ Im $b$.
\item $a \mathscr{R}^\theta b$ if and only if $\pi_a= \pi_b$.
\item $a \mathscr{D}^\theta b$ if and only if rank $a= $ rank $b$.
\end{enumerate}
\end{pro}
Hence from the discussion above, it is clear that the $\mathscr{L}^\theta $-classes of $Reg(\mathscr{T}_X^\theta)$ may be characterized as Im $a$, for $a \in P_1$. Similarly, the $\mathscr{R}^\theta $-classes of $Reg(\mathscr{T}_X^\theta)$ may be characterized as $\pi_a$, for $a \in P_2$. Observe that the $\mathscr{H}$-classes ${H}^\theta_a$ and ${H}_a$ coincide for $a \in Reg(\mathscr{T}_X^\theta)$. So the categories $\mathscr{P}_\theta$ and $\Pi_\theta$ associated with $Reg(\mathscr{T}_X^\theta)$ are the full subcategories of $\mathscr{P}(X)$ and $\Pi(X)$ defined as follows.
\begin{equation*}
v\mathscr{P}_\theta =\{ A : \pi_\theta \text{ separates } A\} \text{ and } v\Pi_\theta =\{ \bar{\pi} : \text{Im }\theta \text{ saturates } \pi \}
\end{equation*}
Then it can be easily shown that $\mathscr{P}_\theta$ and $\Pi_\theta$ form normal categories, and they are respectively the $\mathcal{L}$ and $\mathcal{R}$ categories associated with $Reg(\mathscr{T}_X^\theta)$. Since $\mathscr{P}_\theta$ is a full normal subcategory of $\mathscr{P}(X)$, every normal cone in $T\mathscr{P}_\theta$ will belong to $T\mathscr{P}(X)$, such that $T\mathscr{P}_\theta$ is a regular subsemigroup of $T\mathscr{P}(X)$ \cite{cross}. Hence, every normal cone in $T\mathscr{P}_\theta$ can be represented by a transformation in $\mathscr{T}_X$. In fact, we have the following result.
\begin{pro}
The semigroup $T\mathscr{P}_\theta$ of normal cones in $\mathscr{P}_\theta$ is isomorphic to $P_1$. The semigroup $T\Pi_\theta$ of normal cones in $\Pi_\theta$ is anti-isomorphic to $P_2$. 
\end{pro}
\begin{proof}
First, observe that since $\theta$ is not a permutation, $\pi_\theta$ separates $\{x\}$ for every $x \in X$. Hence all the singletons $\{x\}$ belong to $v\mathscr{P}_\theta$. Any normal cone in $T\mathscr{P}_\theta$ is determined by the action on these singletons. Now given any $A \in v\mathscr{P}_\theta$, every transformation $a$ with Im $a\subseteq A$, will be contained in $T\mathscr{P}_\theta$, and in fact only these transformations. Hence 
$$T\mathscr{P}_\theta = \{ a\in\mathscr{T}_X:\text{ Im } a\in \mathscr{P}_\theta\}  = \{ a\in\mathscr{T}_X:\pi_\theta\text{ separates Im } a\} =P_1.$$
Arguing dually and observing that $T\Pi_\theta$ is a subsemigroup of $\mathscr{T}_X^\text{op}$, we can show that the semigroup $T\Pi_\theta$ is anti-isomorphic to $P_2$.
\end{proof}

Roughly speaking, the semigroup of cones of a normal category may be considered as a kind of universal semigroup with the given ideal structure. So, it is not surprising that $P_1$ and $P_2$ play critical roles in the characterisation of $Reg(\mathscr{T}_X^\theta)$. From the above discussion, it is clear that $P_1$ and $P_2$ are not mere subsets, but important regular semigroups, when seen as subsemigroups of $\mathscr{T}_X$ and $\mathscr{T}_X^\text{op}$ respectively. 

Now, we proceed to characterize the normal dual of the category $\mathscr{P}_\theta$. Observe that an $H$-functor on the category $\mathscr{P}_\theta$ may be represented as $H(a;-)$ for $a\in P_1$. As argued in \cite{tx}, it may be shown that the $H$-functor is completely determined by the partition $\pi_a$ of the transformation $a$. Since $P_1$ contains transformations with all partitions $\pi$ such that $|\pi| \leq \text{ rank } \theta$, we can see that the normal dual of $\mathscr{P}_\theta$ can be characterized as full subcategory of $\Pi(X)$ such that
$$vN^*\mathscr{P}_\theta =\{ \bar{\pi} : |\pi| \leq \text{ rank } \theta\}.$$
Dually, the normal dual of $\Pi_\theta$ can be characterized as the full subcategory of $\mathscr{P}(X)$ such that
$$vN^*\Pi_\theta =\{ A : |A| \leq \text{ rank } \theta \}.$$
To describe the cross-connections in $Reg(\mathscr{T}_X^\theta)$, we will also need the following lemma, whose dual statement is proved in \cite{tx}.
\begin{lem}\label{mset}
Let $a$ represent a normal cone in $\Pi(X)$. Then the $M$-set 
$$Ma \:=  \{ \bar{\pi} \in \Pi(X) :\text{ Im }a \text{ is a cross-section of } \pi \}.$$
\end{lem}

\section{Cross-connections of the variants}\label{secvarcxn}
Having described the normal categories in $Reg(\mathscr{T}_X^\theta)$, now we proceed to show how the categories $\mathscr{P}_\theta$ and $\Pi_\theta$ are {cross-connected} by the sandwich element $\theta$. We shall argue using the cross-connection $\Delta_\theta$ between the subsets, rather than the partitions, since the subset connection is more illuminating and working with the partition category is slightly cumbersome. We can dually extend the argument to $\Gamma_\theta$ as well.

\begin{pro}
For $A\in v\mathscr{P}_\theta$ and $f\colon A \to B$ in $\mathscr{P}_\theta$, a functor $\Delta_\theta\colon  \mathscr{P}_\theta \to N^*\Pi_\theta$ defined as follows is a cross-connection between the categories $\mathscr{P}_\theta$ and $\Pi_\theta$.
\begin{equation*}\label{eqndcrossrp}
\Delta_{\theta} (A) \:=\: A\theta \text{ and } \Delta_{\theta}(f) = (\theta_{|A})^{-1} f \theta 
\end{equation*}
\begin{equation*}\label{Delta}
\xymatrixcolsep{4pc}\xymatrixrowsep{3pc}\xymatrix
{
 A \ar[r]^{\theta} \ar[d]_{f}  
 & A\theta \ar[d]^{(\theta_{|A})^{-1} f \theta } \\       
 B \ar[r]^{\theta} & B\theta 
}
\end{equation*}
\end{pro}  
\begin{proof}
First, we show that $\Delta_{\theta}$ is well-defined. Suppose $A \in \mathscr{P}_\theta$, then $\Delta_{\theta}(1_A) =(\theta_{|A})^{-1} 1_A \theta = 1_{A\theta}$. Also if $f\colon A\to B$ and $g\colon B \to C$, then 
\begin{equation*}
\begin{split}
\Delta_{\theta}(f\circ g) =&  (\theta_{|A})^{-1} (f\circ g) \theta\\
 =& (\theta_{|A})^{-1} f (1_B) g \theta \\ 
 =& (\theta_{|A})^{-1} f(\theta (\theta_{|B})^{-1}) g \theta\\
 =& ((\theta_{|A})^{-1} f\theta) ((\theta_{|B})^{-1} g \theta)\\
 =& \Delta_{\theta}(f)\circ\Delta_{\theta}(g).
\end{split}
\end{equation*}
Observe that $\mathscr{P}_\theta$ is a full subcategory of $N^*\Pi(X)$. Also for $A \in v\mathscr{P}_\theta$, since $\pi_\theta$ separates $A$, $A\mapsto A\theta$ is a bijection. Since $\pi_\theta$ separates $A$, all the subsets of $A$ will also be separated by $\pi_\theta$. Hence, we can see that the functor $\langle A \rangle \mapsto \langle A\theta \rangle$ is a normal category isomorphism. Thus $\Delta_\theta$ is a local isomorphism. Observe here that, $ A\in v\mathscr{P}_\theta$ are precisely those vertices of $\mathscr{P}(X)$, where $\Delta_\theta$ is a local isomorphism.

Since the category $\Pi_\theta$ consists of partitions which are saturated by Im $\theta$, the subsets of Im $\theta$
will be cross-sections of the partitions in $\Pi_\theta$. So for every $\bar{\pi} \in v\Pi_\theta$, there is some $A \in v\mathscr{P}_\theta$ such that $A\theta$ is a cross-section of $\pi$. Thus by Lemma \ref{mset}, for every $\bar{\pi} \in v\Pi_\theta$, there is some $A \in v\mathscr{P}_\theta$ such that $\bar{\pi} \in M\Delta_\theta(A)$. Hence $\Delta_\theta$ is a cross-connection between the categories $\mathscr{P}_\theta$ and $\Pi_\theta$.
\end{proof}
Observe that the cross-connection $\Delta_\theta$ is a {proper} local isomorphism which is not a category isomorphism. Dually, we can define a functor $\Gamma_\theta$ from $\Pi_\theta$ to $\mathscr{P}_\theta$ as follows. For $\bar{\pi} \in v\Pi_\theta$ and $\eta^\ast\colon \bar{\pi}_{1}\to \bar{\pi}_{2}$ in $\Pi_\theta$, 
\begin{equation*}\label{eqncrossrp}
\Gamma_\theta (\bar{\pi}) \:=\: {\theta}^\ast(\bar{\pi}) \text{ and } \Gamma_\theta(\eta^\ast) = (\theta \eta (\theta_{|C})^{-1})^\ast
\end{equation*}
where $C \in \mathscr{P}_\theta$ is a cross-section of $(\pi_1)\theta^{-1}$. Refer \cite{tx} to see the details of how a transformation $\theta$ induces a morphism ${\theta}^\ast$ in the category of partitions from $\bar{\pi}$ to $\overline{(\pi)\theta^{-1}}$ for $\bar{\pi} \in v\Pi_\theta$. Then it can be shown that, $\Gamma_\theta$ is a local isomorphism  such that for every $ A\in \mathscr{P}_\theta$ there exists $\bar{\pi}\in \Pi_\theta$ such that $A\in\Gamma_\theta (\bar{\pi})$. Hence, we can show that $\Gamma_\theta$ is indeed a {proper} cross-connection.

Now, the cross-connections $\Gamma_\theta$ and $\Delta_\theta$ give rise to two bifunctors $\Gamma_\theta(-,-)\colon  \mathscr{P}_\theta \times \Pi_\theta \to \bf{Set}$ and $\Delta_\theta(-,-) \colon  \mathscr{P}_\theta \times \Pi_\theta \to \bf{Set}$ as follows. For all $(A,\bar{\pi}) \in v\mathscr{P}_\theta \times v\Pi_\theta$ and $(f,\eta^\ast)\colon (A,\bar{\pi}_1) \to (B,\bar{\pi}_2)$ 
\begin{subequations}
\begin{align}
\Gamma_\theta (A,\bar{\pi})\:&=\: \{ a \in \mathscr{T}_X : \text{ Im } a \subseteq A \text{ and } \bar{\pi}_{a} \subseteq \theta^*(\bar{\pi}) \}\label{eqngx1}\\
\Gamma_\theta (f,\eta^\ast)\:& : \: a \mapsto \theta \eta (\theta_{|C})^{-1} a f  \label{eqngx2}\\
\Delta_\theta (A,\bar{\pi})\:&=\: \{ a \in \mathscr{T}_X : \text{ Im } a \subseteq \theta(A) \text{ and } \bar{\pi}_{a} \subseteq \bar{\pi} \}\label{eqndx1}\\
\Delta_\theta (f,\eta^\ast)\:& : \: a \mapsto \eta a (\theta_{|A})^{-1}f \theta \label{eqndx2}
\end{align}
\end{subequations}
It may be observed here that although $\eta$ represents a mapping between partitions $\pi_2$ and $\pi_1$, $\eta$ may be restricted to give a mapping between its cross-sections \cite{pix}. The resulting intermediary semigroups $U\Gamma_\theta$ and $U\Delta_\theta$ may be characterized as follows:
\begin{subequations}
\begin{align}
U\Gamma_\theta =& \{a\in \mathscr{T}_X: \pi_a\supseteq\pi_\theta \text{ and Im }a \in v\mathscr{P}_\theta   \} = \{\theta a : a \in Reg(\mathscr{T}_X^\theta) \}\nonumber\\
U\Delta_\theta =& \{a\in \mathscr{T}_X: \pi_a\in v\Pi_\theta \text{ and Im }a \subseteq \text{ Im }\theta   \} = \{ a\theta : a \in Reg(\mathscr{T}_X^\theta) \}\nonumber
\end{align}
\end{subequations}
Observe that the semigroups $U\Gamma_\theta$ and $U\Delta_\theta$ represent the semigroups of principal cones from the categories $\mathcal{L}$ and $\mathcal{R}$ respectively, of the semigroup $Reg(\mathscr{T}_X^\theta)$. Hence, we see that both the categories here have \emph{non-principal} normal cones. In fact, this is a necessary condition for the existence of a `good' cross-connection structure.

Given cross-connections $\Gamma_\theta$ and $\Delta_\theta$, there exists a natural isomorphism $\chi_{\Gamma_\theta}$ between the bifunctors $\Gamma_\theta(-,-)$ and $\Delta_\theta(-,-)$ associated with the cross-connections. This natural isomorphism called the duality of the semigroup $Reg(\mathscr{T}_X^\theta)$ is described in the next proposition.  
\begin{pro}
The duality $\chi_{\Gamma_\theta}\colon \Gamma_\theta(-,-)\to \Delta_\theta(-,-)$ is given by $\chi_{\Gamma_\theta}(A,\bar{\pi})\colon  \theta a \mapsto a \theta$.
\end{pro}
\begin{proof}
First, we show that $\chi_\Gamma$ is indeed a natural transformation. Let $(f,\eta^\ast) \in \mathscr{P}(X) \times \Pi(X)$ where $f \colon  A \to B$ and $\eta^\ast \colon  \bar{\pi}_1 \to \bar{\pi}_2$. 
\begin{equation*}\label{duality}
\xymatrixcolsep{3pc}\xymatrixrowsep{2.5pc}\xymatrix
{
 \Gamma_{\theta }(A,\bar{\pi}_1) \ar[r]^{\chi_{\Gamma_\theta}(A,\bar{\pi}_1)} \ar[d]_{\Gamma_{\theta }(f,\eta^\ast)}  
 & \Delta_{\theta }(A,\bar{\pi}_1) \ar[d]^{\Delta_{\theta }(f,\eta^\ast)} \\       
 \Gamma_{\theta }(B,\bar{\pi}_2) \ar[r]^{\chi_{\Gamma_\theta}(B,\bar{\pi}_2)} & \Delta_{\theta }(B,\bar{\pi}_2) 
}
\end{equation*}
To see that the above diagram commutes, let $ \theta a \in \Gamma_{\theta }(A,\bar{\pi}_1)$. Then,
\begin{align*}
(\theta a)\chi_{\Gamma_\theta}(A,\bar{\pi}_1)\Delta_{\theta }(f,\eta^\ast)& = (  a \theta ) \Delta_{\theta }(f,\eta^\ast)\\
& =  \eta a \theta (\theta_{|A})^{-1}f \theta \text{ (using (\ref{eqndx2}))}\\
& = \eta  a (1_A) f \theta  \text{ (since }\theta (\theta_{|A})^{-1} = 1_A)\\
& = \eta  a f \theta \text{ (since Im }a \subseteq A). 
\end{align*}
\begin{align*}
\text{Also } (\theta a)\Gamma_{\theta }(f,\eta^\ast)\chi_{\Gamma_\theta}(B,\bar{\pi}_2) &= \theta \eta (\theta_{|C})^{-1} \theta a f \chi_{\Gamma_\theta}(B,\bar{\pi}_2) \text{ (using (\ref{eqngx2}))}\\
 &= \theta (\eta (\theta_{|C})^{-1} \theta a f) \chi_{\Gamma_\theta}(B,\bar{\pi}_2)\\
 &= (\eta (\theta_{|C})^{-1} \theta a f)\theta \\
 &= \eta ((\theta_{|C})^{-1} \theta) a f\theta \\
 &= \eta (1_{C \theta } ) a f\theta  \text{ (since }(\theta_{|C})^{-1} \theta = 1_{C \theta })\\
 &= \eta a f\theta \text{ (since }C\theta \text{ is a cross-section of }\pi_1).
\end{align*}
Thus, we have 
$$(\theta a)\chi_{\Gamma_\theta}(A,\bar{\pi}_1)\Delta_{\theta }(f,\eta^\ast) = (\theta a)\Gamma_{\theta }(f,\eta^\ast)\chi_{\Gamma_\theta}(B,\bar{\pi}_2) \:\: \text{ for every } \: \theta a \in \Gamma_{\theta }(A,\bar{\pi}_1).$$ 

Also observe that since the map $\theta a \mapsto  a\theta$ is a bijection for $a \in Reg(\mathscr{T}_X^\theta)$, the map $\chi_{\Gamma_\theta}(A,\bar{\pi})$ is a bijection of the set $\Gamma_{\theta }(A,\bar{\pi})$ onto the set $\Delta_{\theta }(A,\bar{\pi})$. Thus $\chi_{\Gamma_\theta}$ is a natural isomorphism. 
\end{proof}
  
So, an element $\theta  a \in U\Gamma_\theta$ is linked to the element $a \theta \in U\Delta_\theta$. Hence, the resulting cross-connection semigroup is given by 
$$ \tilde{S}\Gamma_\theta = (\mathscr{P}_\theta,\Pi_\theta; \Delta_\theta) = (\Pi_\theta,\mathscr{P}_\theta; \Gamma_\theta) = \{ (\theta a, a \theta) : a \in  Reg(\mathscr{T}_X^\theta)\} .$$
Summarising, we have the following theorem.
\begin{thm}
The cross-connection semigroup $(\Pi_\theta,\mathscr{P}_\theta; \Gamma_\theta)$ is isomorphic to the semigroup $Reg(\mathscr{T}_X^\theta)$ of regular elements of $\mathscr{T}_X^\theta$ such that the $\mathcal{L}$ category of $Reg(\mathscr{T}_X^\theta)$ is isomorphic to $\mathscr{P}_\theta$ and the $\mathcal{R}$ category of $Reg(\mathscr{T}_X^\theta)$ is isomorphic to $\Pi_\theta$.  
\end{thm}

Thus we realise $Reg(\mathscr{T}_X^\theta)$ as a subsemigroup of $\mathscr{T}_X \times \mathscr{T}_X$ so that $Reg(\mathscr{T}_X^\theta) = (\Pi_\theta,\mathscr{P}_\theta; \Gamma_\theta)$ is the result of the categories $\mathscr{P}_\theta$ and $\Pi_\theta$ cross-connected via $\theta$. Observe that, the structural results of Dolinka and East \cite{igd} have a natural cross-connection interpretation. For instance, the map 
$$\psi:Reg(\mathscr{T}_X^\theta) \to Reg(\mathscr{T}(X,A))\times Reg(\mathscr{T}(X,\alpha)):a\mapsto(a\theta,\theta a)$$ being injective, translates to the cross-connection functor being a local isomorphism. So, this dual approach may help in expanding their combinatorial results and also extending the structure theorem to a larger class of semigroups.

Further, observe that the Green relations between the elements are the same in $\mathscr{T}_X$ and in $Reg(\mathscr{T}_X^\theta)$. This illustrates the fact that, when the binary operation is changed, the real structural `variance' is the `cross-connection' of these ideals. Thus, $Reg(\mathscr{T}_X^\theta)$ is a classic case where all the subtleties of cross-connection theory play out nicely, reiterating the necessity of such a sophisticated theory in describing the ideal structure of semigroups.

\section{Biorder structure}

Now, we proceed to describe the biorder structure of $Reg(\mathscr{T}_X^\theta)$ using the cross-connection representation. Recall that, Nambooripad, in his seminal work \cite{mem}, characterised the set of idempotents of a (regular) semigroup as a (regular) biordered set. He gave a structural description of a regular semigroup from its (regular) biordered set, using \emph{inductive groupoids}, exploiting the structural information captured by the idempotents of a semigroup. 

A \emph{biordered set} is a partial algebra whose partial binary operation, called the {basic product}, is determined by two preorders $\omega^l$ and $\omega^r$, and satisfying certain axioms. Given two elements $e,f$ in a biordered set, the \emph{sandwich set} $\mathcal{S}(e,f)$ characterises the inverses (or the regularity) of the elements of the semigroup associated with the biordered set. Roughly speaking, the sandwich set is the (regular semigroup) generalisation of the notion of the meet of two elements in a semilattice (of an inverse semigroup). We refer the reader to \cite{mem,higginstech} for definition and properties of the biordered set and sandwich sets.

In \cite[Section V.1]{cross}, Nambooripad gave the description of biordered set and sandwich sets of a cross-connection semigroup. So, using the cross-connection representation of $Reg(\mathscr{T}_X^\theta)$ obtained in the previous section, we can obtain the biorder structure description of the semigroup, in terms of subsets and partitions. We omit proofs and details, as they may affect the self-containedness of the article. This discussion also suggests that we can always retrieve the idempotent structure from a cross-connection description, reiterating the strength of the theory. 

\subsection{Biordered set and sandwich sets of $\mathscr{T}_X$}

First, in the transformation semigroup $\mathscr{T}_X$, the idempotents are given by
\begin{equation}\label{eqb1}
E(\mathscr{T}_X) = \{(A,\bar{\pi}) \in \mathscr{P}(X)\times\Pi(X) : A \text{ is a cross-section of }\pi \}.
\end{equation}
One can see that for each $(A,\bar{\pi}) \in E(\mathscr{T}_X)$, there exists a unique idempotent transformation $e\in \mathscr{T}_X$ such that Im $ e=A$ and $\pi_e=\pi$. Then $e$ is called the \emph{cross-connection idempotent cone} associated with the idempotent $(A,\bar{\pi})$. Then the preorders $\omega^l$ and $\omega^r$ are defined by:
\begin{equation}\label{eqb2}
(A,\bar{\pi})\omega^l(A',\bar{\pi}')\iff A\subseteq A'\text{  and  }(A,\bar{\pi})\omega^r(A',\bar{\pi}')\iff \bar{\pi}\subseteq \bar{\pi}'.
\end{equation}
For any two idempotents $e=(A,\bar{\pi})$ and $e'=(A',\bar{\pi}')$, we define basic products as:
\begin{equation}\label{eqb3}
(A,\bar{\pi})(A',\bar{\pi}')=
\begin{cases}
    (A,\bar{\pi}), & \text{ if } A\subseteq A';\\
    (A',{\text{Im } {(e_{|A'})^*} } ), & \text{ if } A'\subseteq A;\\
    (A',\bar{\pi}'), & \text{ if } \bar{\pi}'\subseteq \bar{\pi};\\
    (\text{Im }e'_{|A},\bar{\pi}),& \text{ if } \bar{\pi}\subseteq \bar{\pi}'.
  \end{cases}
\end{equation}

Then $E(\mathscr{T}_X)$ as defined in (\ref{eqb1}) forms a regular biordered set with preorders and basic products as defined in (\ref{eqb2}) and (\ref{eqb3}) respectively.

Further, the sandwich set $\mathcal{S}((A,\bar{\pi}'),(A',\bar{\pi}))= \mathcal{S}(A,\bar{\pi}) $ of any two idempotents $(A,\bar{\pi}'), (A',\bar{\pi}) \in E(\mathscr{T}_X)$ is given by, 
\begin{multline}
\mathcal{S}(A,\bar{\pi}) = \{ (X,\bar{\sigma})\in  \mathscr{P}(X)\times\Pi(X)  : X\text{ is a cross-section of }\pi\\
\text{ and } A \text{ is a cross-section of } \sigma \}.
\end{multline}

\begin{rmk}
The biordered set and sandwich sets in the singular transformation semigroup $Sing(X)$ can be obtained from the above description by just restricting the categories $ \mathscr{P}(X)$ and $\Pi(X) $, respectively.
\end{rmk}

\subsection{Biordered set and sandwich sets of $Reg(\mathscr{T}_X^\theta)$}

Now, we proceed to give the biorder structure of $Reg(\mathscr{T}_X^\theta)$. Since $Reg(\mathscr{T}_X^\theta) =(\Pi_\theta,\mathscr{P}_\theta ; \Gamma_\theta)$ , the idempotents $E(Reg(\mathscr{T}_X^\theta)) =E_{\Gamma_\theta}$ of the semigroup $Reg(\mathscr{T}_X^\theta)$ may be described as follows:
\begin{equation}
E_{\Gamma_\theta} = \{(A,\bar{\pi}) \in \mathscr{P}_\theta \times \Pi_\theta: \pi_\theta \text{ separates }A \text{ and } A\theta \text{ is a cross-section of }\pi \}.
\end{equation}
Since $\mathscr{P}_\theta \times \Pi_\theta$ is a full subcategory of $\mathscr{P}(X) \times \Pi(X)$, preorders and basic products defined in (\ref{eqb2}) and (\ref{eqb3}) respectively, restricted to $E_{\Gamma_\theta}$, will give preorders and basic products in $E_{\Gamma_\theta}$. Then it may be verified that $E_{\Gamma_\theta}$ forms a regular biordered set, with quasi-orders and basic products, as defined above.

Now, the sandwich set $ \mathcal{S}(A,\bar{\pi}) =\mathcal{S}((A,\bar{\pi}'),(A',\bar{\pi}))$, where $(A,\bar{\pi}'),(A',\bar{\pi}) \in E_{\Gamma_\theta}$ is given by,
\begin{multline}
\mathcal{S}(A,\bar{\pi})  =\{ (X,\bar{\sigma}) \in \mathscr{P}_\theta \times \Pi_\theta: X\text{ is a cross-section of }\pi\\
 \text{ and } A \text{ is a cross-section of } \sigma \}.
\end{multline}

Although the sandwich set description is independent of the cross-connection, the sandwich sets in $Reg(\mathscr{T}_X^\theta)$ will indeed become smaller than in $\mathscr{T}_X$, because of the restriction in the categories. For instance, $(323) \in \mathcal{S}(\{12\},\overline{\{\{13\}\{2\}\}})=\mathcal{S}((121),(121))$ in $\mathscr{T}_3$; but clearly, $(323) \notin \mathcal{S}((121),(121))$ in $Reg(\mathscr{T}_3^{(122)})$.

\section{An example}

We conclude with an illustrative example. Consider the semigroup $Reg(\mathscr{T}_4^{\theta})$ of regular elements of the variant of the finite full transformation semigroup on a four element set, where $\theta =(1233)$. The egg box diagram of $\mathscr{T}_4^{\theta}$ may be found at \cite[Figure 2]{igd}. First, since $\theta=(1233)$, 
\begin{align*}
v\mathscr{P}_\theta=\{\{123\},\{124\}, \{12\},\{23\},\{24\},&\{14\},\{13\},\{1\},\{2\},\{3\},\{4\} \} ;\\
v\Pi_\theta=\{  \overline{\{1\}\{2\}\{34\}}, \overline{\{14\}\{2\}\{3\}}, \overline{\{1\}\{24\}\{3\}}&, \overline{\{124\}\{3\}} ,  \overline{\{12\}\{34\}} ,  \overline{\{1\}\{234\}} ,  \\ 
&\overline{\{14\}\{23\}},  \overline{\{2\}\{134\}} ,  \overline{\{13\}\{24\} }, \overline{\{1234\}} \}.
\end{align*}

\begin{center}
\begin{tikzpicture}[thick,scale=.9, every node/.style={scale=.9}, node distance=0 cm,outer sep = 0pt]
\tikzstyle{grp}=[draw, rectangle,font={\tiny}, text width= .6cm, minimum height=.7cm, minimum width=1cm, fill=violet!20,anchor=south west]
\tikzstyle{grp0}=[draw, rectangle,font={\tiny}, text width= .6cm, minimum height=.7cm, minimum width=1cm, fill=red!20,anchor=south west]
\tikzstyle{dc}=[draw, rectangle, font={\tiny}, text width= .6cm, minimum height=.7cm, minimum width=1cm, fill=violet!0,anchor=south west]
\tikzstyle{grp1}=[draw, rectangle,font={\tiny}, text width= .6cm,  minimum height=.7cm, minimum width=1cm, fill=blue!20,anchor=south west]
\tikzstyle{dc1}=[draw, rectangle,font={\tiny}, text width= .6cm,  minimum height=.7cm, minimum width=1cm, fill=blue!0,anchor=south west]
\tikzstyle{nm}=[font={\tiny}, text width= .6cm, minimum height=.7cm, minimum width=1cm, fill=white!20,anchor=south west]
\tikzstyle{nt}=[draw, rectangle, text width= 4cm, text centered, minimum height=.7cm, minimum width=3.5cm,fill=green!10, anchor=south west]
\tikzstyle{nc}=[font={\Large}, text centered, minimum height=.7cm, minimum width=1cm, anchor=south west]

\node[nm] (s0a) at (7.2,6.5) {$12$};
\node[nm] (s0b) at (8.2,6.5) {$23$};
\node[nm] (s0c) at (9.2,6.5) {$24$};
\node[nm] (s0d) at (10.2,6.5) {$14$};
\node[nm] (s0e) at (11.2,6.5) {$13$};

\node[nm] (s0) at (11.8,6) {$\overline{\{124\}\{3\}}$};
\node[dc] (s01) at (7,6) {(1121) (2212)};
\node[grp0] (s02) [right = of s01] {(2232) (3323)};
\node[grp0] (s03) [right = of s02] {(2242) (4424)};
\node[grp0] (s04) [right = of s03] {(1141) (4414)};
\node[grp0] (s05) [right = of s04] {(1131) (3313)};

\node[nm] (s1) at (11.8,5.3) {$\overline{\{12\}\{34\}}$};
\node[dc] (s11) [below = of s01] {(1122) (2211)};
\node[grp0] (s12) [right = of s11] {(2233) (3322)};
\node[grp0] (s13) [right = of s12] {(2244) (4422)};
\node[grp0] (s14) [right = of s13] {(1144) (4411)};
\node[grp0] (s15) [right = of s14] {(1133) (3311)};

\node[nm] (s2) at (11.8,4.6) {$\overline{\{1\}\{234\}}$};
\node[grp0] (s21) [below = of s11] {(1222) (2111)};
\node[dc] (s22) [right = of s21] {(2333) (3222)};
\node[dc] (s23) [right = of s22] {(2444) (4222)};
\node[grp0] (s24) [right = of s23] {(1444) (4111)};
\node[grp0] (s25) [right = of s24] {(1333) (3111)};

\node[nm] (s3) at (11.8,3.9) {$\overline{\{14\}\{23\}}$};
\node[grp0] (s31) [below = of s21] {(1221) (2112)};
\node[dc] (s32) [right = of s31] {(2332) (3223)};
\node[dc] (s33) [right = of s32] {(2442) (4224)};
\node[grp0] (s34) [right = of s33] {(1441) (4114)};
\node[grp0] (s35) [right = of s34] {(1331) (3113)};

\node[nm] (s4) at (11.8,3.2) {$\overline{\{2\}\{134\}}$};
\node[grp0] (s41) [below = of s31] {(1211) (2122)};
\node[grp0] (s42) [right = of s41] {(2322) (3233)};
\node[grp0] (s43) [right = of s42] {(2422) (4244)};
\node[dc] (s44) [right = of s43] {(1411) (4144)};
\node[dc] (s45) [right = of s44] {(1311) (3133)};

\node[nm] (s5) at (11.8,2.5) {$\overline{\{13\}\{24\}}$};
\node[grp0] (s51) [below = of s41] {(1212) (2121)};
\node[grp0] (s52) [right = of s51] {(2323) (3232)};
\node[grp0] (s53) [right = of s52] {(2424) (4242)};
\node[dc] (s54) [right = of s53] {(1414) (4141)};
\node[dc] (s55) [right = of s54] {(1313) (3131)};

\node[nc] (g) at (5.7,6.1) {$\Gamma_\theta$};

\node[nm] (l0a) at (.65,6) {$12$};
\node[nm] (l0b) at (1.65,6) {$23$};
\node[nm] (l0c) at (2.65,6) {$24$};
\node[nm] (l0d) at (3.65,6) {$14$};
\node[nm] (l0e) at (4.65,6) {$13$};

\node[nm] (l1) at (-.75,5.5) {$\overline{\{12\}\{34\}}$};
\node[dc] (l11) at (.5,5.5) {(1122) (2211)};
\node[grp] (l12) [right = of l11] {(2233) (3322)};
\node[grp] (l13) [right = of l12] {(2244) (4422)};
\node[grp] (l14) [right = of l13] {(1144) (4411)};
\node[grp] (l15) [right = of l14] {(1133) (3311)};

\node[nm] (l2) at (-.75,4.8) {$\overline{\{1\}\{234\}}$};
\node[grp] (l21) [below = of l11] {(1222) (2111)};
\node[dc] (l22) [right = of l21] {(2333) (3222)};
\node[dc] (l23) [right = of l22] {(2444) (4222)};
\node[grp] (l24) [right = of l23] {(1444) (4111)};
\node[grp] (l25) [right = of l24] {(1333) (3111)};

\node[nm] (l4) at (-.75,4.1) {$\overline{\{2\}\{134\}}$};
\node[grp] (l41) [below = of l21] {(1211) (2122)};
\node[grp] (l42) [right = of l41] {(2322) (3233)};
\node[grp] (l43) [right = of l42] {(2422) (4244)};
\node[dc] (l44) [right = of l43] {(1411) (4144)};
\node[dc] (l45) [right = of l44] {(1311) (3133)};

\node[nm] (l0) at (-.75,3.4) {$\overline{\{124\}\{3\}}$};
\node[dc1] (l01) [below = of l41] {(1121) (2212)};
\node[grp1] (l02) [right = of l01] {(2232) (3323)};
\node[dc1] (l03) [right = of l02] {(2242) (4424)};
\node[dc1] (l04) [right = of l03] {(1141) (4414)};
\node[grp1] (l05) [right = of l04] {(1131) (3313)};

\node[nm] (l3) at (-.75,2.7) {$\overline{\{14\}\{23\}}$};
\node[grp1] (l31) [below = of l01] {(1221) (2112)};
\node[dc1] (l32) [right = of l31] {(2332) (3223)};
\node[grp1] (l33) [right = of l32] {(2442) (4224)};
\node[dc1] (l34) [right = of l33] {(1441) (4114)};
\node[grp1] (l35) [right = of l34] {(1331) (3113)};

\node[nm] (l5) at (-.75,2) {$\overline{\{13\}\{24\}}$};
\node[grp1] (l51) [below = of l31] {(1212) (2121)};
\node[grp1] (l52) [right = of l51] {(2323) (3232)};
\node[dc1] (l53) [right = of l52] {(2424) (4242)};
\node[grp1] (l54) [right = of l53] {(1414) (4141)};
\node[dc1] (l55) [right = of l54] {(1313) (3131)};

\node[nm] (l6) at (-.75,1.3) {$\overline{\{123\}\{4\}}$};
\node[dc1] (l61) [below = of l51] {(1112) (2221)};
\node[dc1] (l62) [right = of l61] {(2223) (3332)};
\node[grp1] (l63) [right = of l62] {(2224) (4442)};
\node[grp1] (l64) [right = of l63] {(1114) (4441)};
\node[dc1] (l65) [right = of l64] {(1113) (3331)};

\node[nc] (d) at (7.5,1.2) {$\Delta_\theta$};

\node[nm] (r0a) at (7.2,-4) {$12$};
\node[nm] (r0b) at (8.2,-4) {$23$};
\node[nm] (r0c) at (9.2,-4) {$13$};
\node[nm] (r0d) at (10.2,-4) {$14$};
\node[nm] (r0e) at (11.2,-4) {$24$};
\node[nm] (r0f) at (12.2,-4) {$34$};

\node[nm] (r0) at (5.75,0) {$\overline{\{124\}\{3\}}$};
\node[nm] (r1) at (5.75,-.7) {$\overline{\{12\}\{34\}}$};
\node[nm] (r2) at (5.75,-1.4) {$\overline{\{1\}\{234\}}$};
\node[nm] (r3) at (5.75,-2.1) {$\overline{\{14\}\{23\}}$};
\node[nm] (r4) at (5.75,-2.8) {$\overline{\{2\}\{134\}}$};
\node[nm] (r5) at (5.75,-3.5) {$\overline{\{13\}\{24\}}$};

\node[dc] (r01) at (7,0) {(1121) (2212)};
\node[grp] (r02) [right = of r01] {(2232) (3323)};
\node[grp] (r03) [right = of r02] {(1131) (3313)};
\node[dc1] (r04) [right = of r03] {(1141) (4414)};
\node[dc1] (r05) [right = of r04] {(2242) (4424)};
\node[grp1] (r06) [right = of r05] {(3343) (4434)};

\node[dc] (r11) [below = of r01] {(1122) (2211)};
\node[grp] (r12) [right = of r11] {(2233) (3322)};
\node[grp] (r13) [right = of r12] {(1133) (3311)};
\node[grp1] (r14) [right = of r13] {(1144) (4411)};
\node[grp1] (r15) [right = of r14] {(2244) (4422)};
\node[dc1] (r16) [right = of r15] {(3344) (4433)};

\node[grp] (r21) [below = of r11] {(1222) (2111)};
\node[dc] (r22) [right = of r21] {(2333) (3222)};
\node[grp] (r23) [right = of r22] {(1333) (3111)};
\node[grp1] (r24) [right = of r23] {(1444) (4111)};
\node[dc1] (r25) [right = of r24] {(2444) (4222)};
\node[dc1] (r26) [right = of r25] {(3444) (4333)};

\node[grp] (r31) [below = of r21] {(1221) (2112)};
\node[dc] (r32) [right = of r31] {(2332) (3223)};
\node[grp] (r33) [right = of r32] {(1331) (3113)};
\node[dc1] (r34) [right = of r33] {(1441) (4114)};
\node[grp1] (r35) [right = of r34] {(2442) (4224)};
\node[grp1] (r36) [right = of r35] {(4334) (3443)};

\node[grp] (r41) [below = of r31] {(1211) (2122)};
\node[grp] (r42) [right = of r41] {(2322) (3233)};
\node[dc] (r43) [right = of r42] {(1311) (3133)};
\node[dc1] (r44) [right = of r43] {(1411) (4144)};
\node[grp1] (r45) [right = of r44] {(2422) (4244)};
\node[dc1] (r46) [right = of r45] {(4344) (3433)};

\node[grp] (r51) [below = of r41] {(1212) (2121)};
\node[grp] (r52) [right = of r51] {(2323) (3232)};
\node[dc] (r53) [right = of r52] {(1313) (3131)};
\node[grp1] (r54) [right = of r53] {(1414) (4141)};
\node[dc1] (r55) [right = of r54] {(2424) (4242)};
\node[grp1] (r56) [right = of r55] {(4343) (3434)};

\draw[->][line width=1pt] (s01.west) to (l15.east);
\draw[->][line width=1pt] (s11.west) to (l15.east);
\draw[->][line width=1pt] (s21.west) to (l25.east);
\draw[->][line width=1pt] (s31.west) to (l25.east);
\draw[->][line width=1pt] (s41.west) to (l45.east);
\draw[->][line width=1pt] (s51.west) to (l45.east);

\draw[->][line width=1pt] (s51.south) to  (r01.north);
\draw[->][line width=1pt] (s52.south) to  (r02.north);
\draw[->][line width=1pt] (s53.south) to  (r02.north);
\draw[->][line width=1pt] (s54.south) to  (r03.north);
\draw[->][line width=1pt] (s55.south) to  (r03.north);

\draw[|-|] (.5,0.5) -- node[fill=white,inner sep=1mm,midway] {$T\mathscr{P}_\theta$}(5.5,0.5);
\draw[|-|] (-1.,6.2) -- node[fill=white,inner sep=1mm,midway] {$U\Gamma_\theta$}(-1.,4.1);
\draw[|-|] (5.55,0.7) -- node[fill=white,inner sep=1mm,midway] {$T\Pi_\theta$}(5.55,-3.5);
\draw[|-|] (7,-4.) -- node[fill=white,inner sep=1mm,midway] {$U\Delta_\theta$}(10,-4.);
\node[nt] (note) at (0,-2.3) {$\mathscr{D}^{\theta}$-class of rank two in $Reg(\mathscr{T}_4^{\theta})$, where $\theta=(1233)$.};

\end{tikzpicture}
\end{center}

The above diagram illustrates the cross-conection structure of the regular $\mathscr{D}^\theta$-class of rank two in $Reg(\mathscr{T}_4^{\theta})$. The coloured blocks indicate group $\mathscr{H}$-classes and $\mathscr{H}^\theta$-classes. 

The top right box represents the $\mathscr{D}^\theta$-class of rank two. The columns and rows here represent the $\mathscr{L}^\theta$-classes and $\mathscr{R}^\theta$-classes in $Reg(\mathscr{T}_4^{\theta})$, respectively. Hence, they denote the object sets of the categories $\mathscr{P}_\theta$ and $\Pi_\theta$, respectively.

First, observe that the idempotents in $\mathscr{T}_4$ and  $\mathscr{T}_4^\theta$ need not coincide. For instance, $(2242) \in E(\mathscr{T}_4^{(1233)})$ but $(2242) \notin E(\mathscr{T}_4)$. On the contrary, $(1414) \in E(\mathscr{T}_4)$ but $(1414) \notin E(\mathscr{T}_4^{(1233)})$.

The top left box represents the corresponding the $\mathscr{D}$-class of rank two in the regular semigroup $T\mathscr{P}_\theta$ of normal cones from the category $\mathscr{P}_\theta$. Recall that $T\mathscr{P}_\theta$  is isomorphic to $P_1$, seen as a subsemigroup of $\mathscr{T}_X$. Observe that, only the top three rows of $T\mathscr{P}_\theta$ (violet coloured part in the pdf file) are the principal cones, and they give rise to linked cones. Thus, in this case, the regular semigroup $U\Gamma_\theta \subsetneq T\mathscr{P}_\theta$. 

From the diagram, it is clear how $\Gamma_\theta$ is a local isomorphism from the category $\mathcal{R}(Reg(\mathscr{T}_4^{\theta}))$ to $\mathcal{R}(T\mathscr{P}_\theta)\cong N^*\mathcal{L}(Reg(\mathscr{T}_4^{\theta}))$. It is indeed a proper local isomorphism; for instance, the first two rows ($\mathscr{R}^\theta$-classes) of $Reg(\mathscr{T}_4^{\theta})$ get mapped to the first row ($\mathscr{R}$-class) of $T\mathscr{P}_{\theta}$.

Similarly, the bottom right box represents the corresponding $\mathscr{D}$-class of rank two in the regular semigroup $T\Pi_\theta$ of normal cones from the category $\Pi_\theta$. The semigroup $T\Pi_\theta$  is isomorphic to $P_2$, seen as the subsemigroup of $\mathscr{T}_X^\text{op}$. Here, the first three columns from the left (violet coloured part in the pdf file) are the principal cones in $T\Pi_\theta$, forming the regular semigroup $U\Delta_\theta\subsetneq T\Pi_\theta$, such that $\Delta_\theta$ is the cross-connection functor.

\section*{Acknowledgements}
\noindent The author is grateful to A. R. Rajan, University of Kerala, India for several fruitful discussions during the preparation of this article.\\
The author expresses his deep gratitude to M. V. Volkov, Ural Federal University, for his helpful advice, suggestions and remarks.\\

\bibliographystyle{plain}
\bibliography{variant}

\end{document}